\documentclass[10pt]{amsart}

\usepackage{latexsym}
\usepackage{amssymb}
\usepackage{amsmath}

\newtheorem{theorem}{Theorem}[section]
\newtheorem{lemma}[theorem]{Lemma}
\newtheorem{proposition}[theorem]{Proposition}
\newtheorem{corollary}[theorem]{Corollary}

\theoremstyle{definition}
\newtheorem{definition}[theorem]{Definition}

\theoremstyle{remark}

\def\N{\mathbb{N}}
\def\Z{\mathbb{Z}}
\def\R{\mathbb{R}}

\def\hN{{{}^*\N}}
\def\hZ{{{}^*\Z}}
\def\hR{{{}^*\R}}
\def\hA{{{}^*\!A}}
\def\hB{{{}^*\!B}}

\def\hY{{{}^*Y}}

\def\B{\mathcal{B}}
\def\D{\mathcal{D}}

\def\U{\mathcal{U}}
\def\V{\mathcal{V}}

\begin{document}

\title[Embeddability properties of difference sets]
{Embeddability properties of
\\
difference sets}

\author{Mauro Di Nasso}

\address{Dipartimento di Matematica\\
Universit\`a di Pisa, Italy}

\email{dinasso@dm.unipi.it}

\date{}

\begin{abstract}
By using nonstandard analysis, we prove embeddability properties
of difference sets $A-B$ of sets of integers.
(A set $A$ is ``embeddable" into $B$ if every finite configuration
of $A$ has shifted copies in $B$.)
As corollaries of our main theorem, we obtain
improvements of results by I.Z.~Ruzsa about intersections
of difference sets, and of Jin's theorem (as refined
by V.~Bergelson, H.~F\"urstenberg and B.~Weiss), where a precise bound is given
on the number of shifts of $A-B$ which are needed to
cover arbitrarily large intervals.
\end{abstract}



\maketitle

\section*{Introduction}
In several areas of combinatorics of numbers, diverse non-elementary techniques
have been successfully used, including ergodic theory, Fourier analysis,
(discrete) topological dynamics, and algebra on the space
of ultrafilters (see \emph{e.g.} \cite{fur,be2,hs,tv,gr} and reference therein).
Also \emph{nonstandard analysis} has been applied in this context, starting
from some early work appeared in the last years of the 80s
(see \cite{hir,let}), and recently producing interesting results
in density problems (see \emph{e.g.} \cite{jin2,jin3,jin4}).

\smallskip
An important topic in combinatorics of numbers is the study
of \emph{sumsets} and of \emph{difference sets}. In 2000, R.~Jin
\cite{jin1} proved by nonstandard methods the following beautiful
property: If $A$ and $B$ are sets of natural numbers
with positive \emph{upper Banach density} then the corresponding
sumset $A+B$ is \emph{piecewise syndetic}.
(A set $C$ is piecewise syndetic if it has ``bounded gaps" in arbitrarily
long intervals; equivalently, if a suitable finite union of
shifts $C+x_i$ covers arbitrarily long intervals.
The upper Banach density is a refinement
of the upper asymptotic density. See below for precise definitions.)

\smallskip
Jin's result raised the attention of several researchers, who tried to
translate his nonstandard proof into more familiar terms, and to
improve on it. In 2006, by using ergodic-theoretical tools,
V.~Bergelson, H.~Furstenberg and B.~Weiss \cite{bfw} gave a new proof
by showing that the set $A+B$ is in fact \emph{piecewise Bohr},
a property stronger than piecewise syndeticity.
In 2008, V.~Bergelson, M.~Beiglb\"ock and A.~Fish
found a shorter proof of that theorem, and extended its validity
to countable \emph{amenable groups}. They also showed that
a converse result holds, namely that every piecewise
Bohr set includes a sumset $A+B$ for suitable
sets $A,B$ of positive density (see \cite{bbf}).
This result was then extended by J.T.~Griesmer \cite{gri}
to cases where one of the summands has zero upper Banach density.
In 2010, M.~Beiglb\"ock \cite{bei} found a very short and
neat \emph{ultrafilter} proof of the afore-mentioned piecewise Bohr property.

\smallskip
In this paper, we work in the setting of the \emph{hyperintegers} of
nonstandard analysis, and we prove some ``embeddability properties" of
sets of differences.
(Clearly, general results on sumsets of natural numbers
directly follow from the corresponding general results on
difference sets, as a set $A$ and the set of its opposites
$-A$ have the same upper Banach density).
A set $A$ is ``embeddable" into $B$ if
every finite configuration of $A$ has shifted copies in $B$,
so that the finite combinatorial structure of $B$ is at least as
rich as that of $A$.
As corollaries to our main theorem, we obtain at once improvements
of results by I.Z.~Rusza about intersections of difference sets, and a
sharpening of Jin's theorem (as refined by V.~Bergelson,
H.~F\"urstenberg and and B.~Weiss).
We remark that many of the results proved here for sets of integers
can be generalized to amenable groups (see \cite{dl}).

\smallskip
The first section of this paper contains the basic notions and
notation, and the statements of the main results.
In the second section, characterizations of several
combinatorial notions in the nonstandard setting of hyperintegers
are presented, which will be used in the sequel.
Section 3 is focused on \emph{Delta sets} $A-A$ and, more generally,
on density-Delta sets. In the next fourth section, we isolate
notions of \emph{finite embeddability} for sets of integers,
and show their basic properties.
The main results of this paper about difference sets $A-B$,
along with several corollaries, are proved in the last Section 5.

\bigskip
\section{Preliminaries and statement of the main results}\label{sec-preliminaries}

If not specified otherwise, throughout the paper by ``set" we shall
always mean a set of \emph{integers}.
By the set $\N$ of natural numbers we mean the set
of \emph{positive integers},
so that $0\notin\N$.

\smallskip
Recall the following basic definitions (see \emph{e.g.} \cite{tv}).
The \emph{difference set} and the \emph{sumset} of $A$ and $B$
are respectively:
$$A-B\ =\ \{a-b\mid a\in A,\ b\in B\}\ ;\quad
A+B\ =\ \{a+b\mid a\in A,\ b\in B\}.$$

The set of differences $\Delta(A)=A-A$ when the two sets are
the same, is called the \emph{Delta set} of $A$.
Clearly, Delta sets are symmetric around $0$, \emph{i.e.}
$t\in\Delta(A)\Leftrightarrow -t\in\Delta(A)$.\footnote
{~Some authors include in $\Delta(A)$ only the positive numbers of $A-A$.}


\smallskip
A set is \emph{thick} if it includes arbitrarily
long (finite) intervals; it is \emph{syndetic} if it has bounded gaps,
\emph{i.e.} if its complement is not thick; it is
\emph{piecewise syndetic} if $A=B\cap C$ where $B$ is syndetic and $C$ is thick.
The following characterizations
directly follow from the definitions:
$A$ is syndetic if and only if
$A+F=\Z$ for a suitable finite set $F$;
$A$ is piecewise syndetic if and only if
$A+F$ is thick for a suitable finite set $F$.

\smallskip
The \emph{lower asymptotic density} $\underline{d}(A)$ and
the \emph{upper asymptotic density} $\overline{d}(A)$ of a set $A$
of natural numbers are defined by putting:
$$\underline{d}(A)\ =\ \liminf_{n\to\infty}\frac{|A\cap[1,n]|}{n}\quad\text{and}\quad
\overline{d}(A)\ =\ \limsup_{n\to\infty}\frac{|A\cap[1,n]|}{n}\,.$$

Another notion of density for sets of natural numbers that is widely used in
number theory is \emph{Schnirelmann density}:
$$\sigma(A)\ =\ \inf_{n\in\N}\frac{|A\cap[1,n]|}{n}.$$

The \emph{upper Banach density} $\text{BD}(A)$ (also known as \emph{uniform density})
generalizes the upper density by considering arbitrary
intervals in place of initial intervals:
$$\text{BD}(A)\ =\
\lim_{n\to\infty}\left(\max_{x\in\Z}\frac{|A\cap[x+1,x+n]|}{n}\right)\ =\
\inf_{n\in\N}\left\{\max_{x\in\Z}\frac{|A\cap[x+1,x+n]|}{n}\right\}.$$

We shall also consider
the \emph{lower Banach density}:
$$\underline{\text{BD}}(A)\ =\
\lim_{n\to\infty}\left(\min_{x\in\Z}\frac{|A\cap[x+1,x+n]|}{n}\right)\ =\
\sup_{n\in\N}\left\{\min_{x\in\Z}\frac{|A\cap[x+1,x+n]|}{n}\right\}.$$
(See \emph{e.g.} \cite{gtt} for details about equivalent definitions of Banach density.)
All the above densities are \emph{shift invariant}, that is
a set $A$ has the same density of any shift $A+t$ .
It is readily verified that $\sigma(A)\le\underline{d}(A)$ and that
$$\underline{\text{BD}}\,(A)\ \le\ \underline{d}(A)\ \le\
\overline{d}(A)\ \le\ \text{BD}(A).\footnote
{~Actually, for any choice of real numbers $0\le r_1\le r_2\le r_3\le r_4\le 1$,
it is not hard to find sets $A$ such that
$\underline{\text{BD}}(A)=r_1$, $\underline{d}(A)=r_2$,
$\overline{d}(A)=r_3$ and $\text{BD}(A)=r_4$.}$$

Notice that $\underline{d}(A^c)=1-\overline{d}(A)$ and
$\underline{\text{BD}}(A^c)=1-\text{BD}(A)$.
Remark also that a set $A$ is thick if and
only if $\text{BD}(A)=1$, and hence, a set $A$ is syndetic
if and only if $\underline{\text{BD}}(A)>0$.
The following is a well-known intersection property of Delta sets.

\medskip
\begin{proposition}\label{pigeonhole}
Assume that $\text{BD}(A)>0$. Then $\Delta(A)\cap\Delta(B)\ne\emptyset$
for \emph{any} infinite set $B$.
In consequence, $\Delta(A)$ is syndetic.
\end{proposition}

\begin{proof}
It essentially consists in a direct application
of the \emph{pigeonhole principle} argument.
Precisely, one considers the family of shifts $\{A+b_i\mid i=1,\ldots,n\}$
by distinct elements $b_i\in B$. As each $A+b_i$ has the same upper Banach
density as $A$, if $n$ is sufficiently large then those shifts cannot
be pairwise disjoint, as otherwise
$\text{BD}(\bigcup_{i=1}^n A+b_i)=\sum_{i=1}^n\text{BD}(A+b_i)=
n\cdot\text{BD}(A)>1$, a contradiction. But then
$(A+b_i)\cap(A+b_j)\ne\emptyset$ for suitable $i\ne j$, and hence
$\Delta(A)\cap\Delta(B)\ne\emptyset$, as desired.
Now assume by contradiction that the complement
of $\Delta(A)$ is thick. By symmetry, its positive
part $T=\Delta(A)^c\cap\N$ is thick as well.
For any thick set $T\subseteq\N$, it is not hard
to construct an increasing sequence $B=\{b_1<b_2<\ldots\}$
such that $b_j-b_i\in T$ for all $j>i$.
But then $\Delta(B)\subseteq -T\cup\{0\}\cup T=\Delta(A)^c$,
\emph{i.e.} $\Delta(B)\cap\Delta(A)=\emptyset$, a contradiction.
\end{proof}

\smallskip
The above property is just a hint of the rich combinatorial structure of
sets of differences, whose investigation seems still far
to be completed (see \emph{e.g.} the recent papers \cite{rs,br,lm}).

\smallskip
Suitable generalizations of Delta sets are the following.

\medskip
\begin{definition}\label{def-finiteembeddability}
Let $A\subseteq\Z$ be a set of integers. For $\epsilon\ge 0$,
the following are called the \emph{$\epsilon$-density-Delta sets}
(or more simply \emph{$\epsilon$-Delta sets}) of $A$:

\smallskip
\begin{itemize}
\item
$\overline{\Delta}_\epsilon(A)\ =\ \{t\in\Z\mid\overline{d}(A\cap(A-t))>\epsilon\}$.

\smallskip
\item
$\Delta_\epsilon(A)\ =\ \{t\in\Z\mid\text{BD}(A\cap(A-t))>\epsilon\}$.
\end{itemize}
\end{definition}

\medskip
Similarly to Delta sets, also $\epsilon$-Delta sets are symmetric around $0$.
Moreover, it is readily seen that
$\overline{\Delta}_\epsilon(A)\subseteq\Delta_\epsilon(A)\subseteq\Delta(A)$
for all $\epsilon\ge 0$. Remark that if $t\in\Delta_\epsilon(C)$
(or if $t\in\overline{\Delta}_\epsilon(C)$), then $t$ is indeed the common
difference of ``many" pairs of elements of $A$, in the sense that
the set $\{x\in\Z\mid x, x+t\in A\}$ has upper Banach density
(or upper asymptotic density, respectively) greater than $\epsilon$.

\smallskip
We shall find it convenient to isolate the following notions
of embeddability for sets of integers $X,Y$.

\medskip
\begin{definition}\label{def-epsilondeltasets}
Let $X,Y\subseteq\Z$ be sets of integers.

\smallskip
\begin{itemize}
\item
$X$ is \emph{(finitely) embeddable} in $Y$, denoted $X\lhd Y$,
if every finite configuration $F\subseteq X$
has a shifted copy $t+F\subseteq Y$.

\smallskip
\item
$X$ is \emph{densely embeddable} in $Y$,
denoted $X\lhd_d Y$, if every finite configuration $F\subseteq X$
has ``densely-many'' shifted copies included in $Y$,
\emph{i.e.} if the intersection
$\bigcap_{x\in F}(Y-x)=\{t\in\Z\mid t+F\subseteq Y\}$
has positive upper Banach density.
\end{itemize}
\end{definition}

\medskip
Trivially $X\lhd_d Y \Rightarrow X\lhd Y$, and it is easily
seen that the converse implication does not hold.
Finite embeddability preserves several of the fundamental
combinatorial notions that are commonly considered in combinatorics
of integer numbers (see Section \ref{sec-embeddability}).\footnote
{~The notions of embeddability isolated above seem to be of interest for their own sake.
\emph{E.g.}, one can extend finite embeddability to \emph{ultrafilters} on $\N$,
by putting $\U\lhd \V$ when for every $B\in\V$ there exists $A\in\U$
with $A\lhd B$. The resulting relation in the space of ultrafilters $\beta\N$
satisfies several nice properties, which are investigated in \cite{bd}.}

\smallskip
The main results obtained in this paper are contained in
the following three theorems. The first one is about the
syndeticity property of $\epsilon$-Delta sets.

\bigskip
\noindent
\textbf{Theorem I.}
\emph{Let $\text{BD}(A)=\alpha>0$ (or $\overline{d}(A)=\alpha>0$), and
let $0\le\epsilon<\alpha^2$. Then for every infinite $X\subseteq\Z$
and for every $x\in X$ there exists a finite subset $F\subset X$ such that:}

\begin{enumerate}
\item
$x\in F$;

\smallskip
\item
$|F|\le\lfloor\frac{\alpha-\epsilon}{\alpha^2-\epsilon}\rfloor=k$;

\smallskip
\item
$X\subseteq\Delta_\epsilon(A)+F$
(or $X\subseteq\overline{\Delta}_\epsilon(A)+F$, respectively).
\end{enumerate}
\emph{In consequence, the set $\Delta_\epsilon(A)$
(or $\overline{\Delta}_\epsilon(A)$, respectively) is syndetic,
and its lower Banach density is not smaller than $1/k$.}

\bigskip
The second theorem is a general property that
holds for all sets of positive upper Banach density.

\bigskip
\noindent
\textbf{Theorem II.}
\emph{Let $\text{BD}(A)=\alpha>0$. Then there exists
a set of natural numbers $E\subseteq\N$ such that:}

\begin{enumerate}
\item
$\sigma(E)\ge\alpha$;

\smallskip
\item
\emph{$E\lhd_d A$, and hence $\Delta(E)\subseteq\Delta_0(A)$
and $\Delta_\epsilon(E)\subseteq\Delta_\epsilon(A)$ for all $\epsilon\ge 0$.}
\end{enumerate}

\medskip
The main result in this paper concerns an embeddability property of
difference sets.

\bigskip
\noindent
\textbf{Theorem III.}
\emph{Let $\text{BD}(A)=\alpha>0$ and $\text{BD}(B)=\beta>0$.
Then there exists a set of natural numbers $E\subseteq\N$ such that:}

\smallskip
\begin{enumerate}
\item
\emph{The Schnirelmann density $\sigma(E)\ge\alpha\beta$;}

\smallskip
\item
\emph{For every finite $F\subset E$ there exists $\epsilon>0$
such that for arbitrarily large intervals $J$
one finds a suitable shift $A_J=A-t_J$ with the property that}
$$\frac{|\left(\bigcap_{e\in F}(A_J\cap B)-e\right)\cap J|}{|J|}\ \ge\ \epsilon$$

\smallskip
\item
\emph{Both $E\lhd_d A$ and $E\lhd_d B$, and hence:}

\smallskip
\begin{itemize}
\item
\emph{$\Delta(E)\subseteq\Delta_0(A)\cap\Delta_0(B)$;}

\smallskip
\item
\emph{$\Delta_\epsilon(E)\subseteq\Delta_\epsilon(A)\cap\Delta_\epsilon(B)$ for all $\epsilon\ge 0$;}

\smallskip
\item
\emph{$\Delta(E)\lhd_d A-B$.}
\end{itemize}
\end{enumerate}

\bigskip
Several corollaries can be derived from the above theorems.
The first one is a sharpening of a result about intersections of
Delta sets by I.Z.~Ruzsa \cite{ru2}, which improved
on a previous theorem by C.L.~Stewart and R.~Tijdeman \cite{st1}.

\bigskip
\noindent
\textbf{Corollary.}
\emph{Assume that $A_1,\ldots,A_n\subseteq\Z$ have
positive upper Banach densities $\text{BD}(A_i)=\alpha_i$.
Then there exists a set $E\subseteq\N$ with
$\sigma(E)\ge\prod_{i=1}^n\alpha_i$ and such that
$\Delta_\epsilon(E)\subseteq\bigcap_{i=1}^n\Delta_\epsilon(A_i)$
for every $\epsilon\ge 0$.}

\bigskip
A second corollary is about the syndeticity of
intersections of density-Delta sets.

\bigskip
\noindent
\textbf{Corollary.}
\emph{Assume that $\text{BD}(A)=\alpha>0$ and $\text{BD}(B)=\beta>0$. Then
for every $0\le\epsilon<\alpha^2\beta^2$,
for every infinite $X\subseteq\Z$, and for every $x\in X$,
there exists a finite subset $F\subset X$ such that}

\begin{enumerate}
\item
$x\in F$\,;

\smallskip
\item
$|F|\le\lfloor\frac{\alpha\beta-\epsilon}{\alpha^2\beta^2-\epsilon}\rfloor=k$\,;

\smallskip
\item
$X\subseteq (\Delta_\epsilon(A)\cap\Delta_\epsilon(B))+F$.
\end{enumerate}
\emph{In consequence, the set $\Delta_\epsilon(A)\cap\Delta_\epsilon(B)$
is syndetic, and its lower Banach density is not smaller than $1/k$.}

\bigskip
A similar result is also obtained about the syndeticity of
difference sets.

\bigskip
\noindent
\textbf{Corollary.}
\emph{Assume that $\text{BD}(A)=\alpha>0$ and $\text{BD}(B)=\beta>0$.
Then for every infinite $X\subseteq\Z$ and
for every $x\in X$, there exists a finite subset $F\subset X$ such that}

\begin{enumerate}
\item
$x\in F$\,;

\smallskip
\item
$|F|\le\lfloor\frac{1}{\alpha\beta}\rfloor$\,;

\item
$X\lhd_d (A-B)+F$.
\end{enumerate}

\bigskip
If we let $X=\Z$, we obtain a refinement of Jin's theorem \cite{jin1}
where a precise bound on the number of shifts of $A-B$ which are
needed to cover a thick set is given.

\bigskip
\noindent
\textbf{Corollary.}
\emph{Assume that $\text{BD}(A)=\alpha>0$ and $\text{BD}(B)=\beta>0$.
Then there exists a finite set $F$ such that
$|F|\le\lfloor 1/{\alpha\beta}\rfloor$
and $A-B+F$ is thick.}

\bigskip
Finally, by the embedding $\Delta(E)\lhd_d A-B$
where $E$ has a positive Schnirelmann density,
we can also recover the \emph{Bohr property} of difference sets proved
by V.~Bergelson, H.~F\"urstenberg and B.~Weiss in \cite{bfw}.

\bigskip
\noindent
\textbf{Corollary.}
\emph{Let $A$ and $B$ have positive upper Banach
density. Then the difference set $A-B$ is piecewise Bohr.}

\bigskip
\section{Nonstandard characterizations of combinatorial properties}

In the proofs of this paper we shall use the basics of \emph{nonstandard analysis},
including the \emph{transfer principle} and the notion of \emph{internal set}
and of \emph{hyperfinite set}.
In particular, the reader is assumed to be familiar with the
fundamental properties of the \emph{hyperintegers} $\hZ$ and of
the \emph{hyperreals} $\hR$.
The hyperintegers are special \emph{elementary extensions} of the
integers, namely \emph{complete extensions} (See \S 3.1 and 6.4 of
\cite{ck} for the definitions).
Informally, one could say that the hyperintegers are a
sort of ``weakly isomorphic extensions"
of the integers, in the sense that they share the same
``elementary" (\emph{i.e.} first-order) properties of $\Z$; in particular, $\hZ$ is a discretely
ordered ring whose positive part is the set $\hN$ of \emph{hypernatural} numbers.
Recall that the natural numbers $\N$ are an initial segment of $\hN$.
Similarly, the hyperreal numbers $\hR\supset\R$ have
the same first-order properties as the reals,
and so they are an ordered field.
As a proper extension of the real line,
$\hR$ is necessarily non-Archimedean, and hence
it contains \emph{infinitesimal} and \emph{infinite} numbers.
Recall that a number $\xi\in\hR$ is infinitesimal if $-1/n<\xi<1/n$ for all
$n\in\N$; $\xi$ is infinite if its reciprocal $1/\xi$ is
infinitesimal, \emph{i.e.} if $|\xi|>n$  for all $n\in\N$;
$\xi$ is finite if it is not infinite, \emph{i.e.} if $-n<\xi<n$ for some $n\in\N$.
In one occasion (proof of Proposition \ref{fe}), we shall apply
the \emph{overspill principle}, namely the property that if an internal set
contains arbitrarily large (finite) natural numbers, then it necessarily contains
also an \emph{infinite} hypernatural number.

\smallskip
A semi-formal introduction to the basic ideas of nonstandard analysis
can be found in the first part of the survey \cite{bdf};
as for the general theory, several books can be used as references,
including the classical monographies \cite{sl,kei},
or the more recent textbook \cite{gol}; finally, we refer the reader
to \S 4.4 of \cite{ck} for the logical foundations.

\smallskip
Let us now fix our notation.
If $\xi,\zeta\in\hR$ are hyperreal numbers, we write
$\xi\approx\zeta$ when $\xi$ and $\zeta$ are \emph{infinitely close},
\emph{i.e.} when their distance $|\xi-\zeta|$ is infinitesimal.
If $\xi\in\hR$ is finite,
then its \emph{standard part} $\text{st}(\xi)=\inf\{r\in\R\mid r>\xi\}$
is the unique real number which is infinitely close to $\xi$.
For $x\in\R$, $\lfloor x\rfloor=\max\{k\in\Z\mid k\le x\}$
is the \emph{integer part} of $x$; and the same notion transfers to
the \emph{hyperinteger part}
of an hyperreal number $\lfloor \xi\rfloor=\max\{\nu\in\hZ\mid \nu\le\xi\}$.
Also the notion of sumset $C+D$ and difference set $C-D$ transfers to
pairs of internal sets $C,D\subseteq\hZ$.
If $C$ is a hyperfinite set, we shall abuse notation and
denote by $|C|\in\hN$ its \emph{internal cardinality}.
An infinite interval of hyperintegers is an
interval $I=[\Omega+1,\Omega+N]\subset\hZ$ whose length $N$ is an
infinite hypernatural number. Clearly, the internal cardinality $|I|=N$.

\smallskip
We shall use the following nonstandard characterizations
(see \emph{e.g.} \cite{jin1,jin2}).

\smallskip
\begin{itemize}
\item
$A$ is \emph{thick} $\Leftrightarrow I\subseteq\hA$
for some infinite interval $I$ of hyperintegers.

\smallskip
\item
$A$ is \emph{syndetic} $\Leftrightarrow$
$\hA$ has only finite gaps, \emph{i.e.}\,\,the distance
of consecutive elements of $\hA$ is always a (finite)
natural number.

\smallskip
\item
$A$ is \emph{piecewise syndetic} $\Leftrightarrow$
there is an infinite interval $I$ of hyperintegers
where $\hA$ has only finite gaps.

\smallskip
\item
$\underline{d}(A)\le\alpha$ (or $\overline{d}(A)\ge\alpha$)
$\Leftrightarrow$ there is an infinite hypernatural number $N$
such that $\text{st}({|\hA\cap[1,N]|}/{N})\le\alpha$
(or $\text{st}({|\hA\cap[1,N]|}/{N})\ge\alpha$, respectively).

\smallskip
\item
$d(A)=\alpha$ $\Leftrightarrow$ ${|\hA\cap[1,N]|}/{N}\approx\alpha$
for all infinite $N$.

\smallskip
\item
$\text{BD}(A)\ge\alpha$ $\Leftrightarrow$
there exists an infinite interval of hyperintegers $I\subset\hZ$
such that $\text{st}({|\hA\cap I|}/{|I|})\ge\alpha$
$\Leftrightarrow$ for every infinite $N\in\hN$
there exists an interval $I\subset\hZ$ of
length $N$ such that
$\text{st}({|\hA\cap I|}/{|I|})\ge\alpha$.

\smallskip
\item
$\underline{\text{BD}}(A)\ge\alpha$ $\Leftrightarrow$
$\text{st}({|\hA\cap I|}/{|I|})\ge\alpha$
for every infinite interval of hyperintegers $I\subset\hZ$.
\end{itemize}

\medskip
As a warm-up for the use of the above nonstandard
characterizations, let us prove a property
which will be used in the sequel.

\medskip
\begin{proposition}\label{lowerbanach}
Let $A$ be a set of integers and let
$F$ be a finite set with $|F|=k$.

\begin{enumerate}
\item
If $A+F=\Z$ then $\underline{\text{BD}}\,(A)\ge 1/k$.

\smallskip
\item
If $A+F$ is thick then $\text{BD}(A)\ge 1/k$.
\end{enumerate}
\end{proposition}

\begin{proof}
$(1)$. For every interval $I$
of infinite length $N$, we have that:
$$I\ =\ \hZ\cap I\ =\ \ {}^*\!\left(\bigcup_{x\in F}
(A+x)\right)\cap I\ =\
\bigcup_{x\in F}\left((\hA+x)\cap I\right).$$
By the \emph{pigeonhole principle}, there exists
$\overline{x}\in F$ such that $|(\hA+\overline{x})\cap I|\ge |I|/k$,
and hence $\text{st}\left(|\hA\cap I|/|I|\right)=
\text{st}\left(|(\hA+\overline{x})\cap I|/|I|\right)\ge 1/k$.
By the nonstandard characterization of lower Banach density,
this yields the thesis $\underline{\text{BD}}(A)\ge 1/k$.

\smallskip
$(2)$. By the nonstandard characterization of thickness,
there exists an infinite interval $I$ with
$I\subseteq{}^*(A+F)=\bigcup_{x\in F}(\hA+x)$.
Exactly as above, we can pick an element $\overline{x}\in F$ such that
$|(\hA+\overline{x})\cap I|\ge |I|/k$, and hence
$\text{st}\left(|\hA\cap I|/|I|\right)\ge 1/k$.
By the nonstandard characterization of Banach density,
we conclude that $\text{BD}(A)\ge 1/k$.
\end{proof}

\bigskip
\section{Density-Delta sets}\label{sec-densitydelta}

In Section \ref{sec-preliminaries}, we recalled the
well-known property
that all intersections of Delta sets
$\Delta(A)\cap\Delta(B)$ are non-empty,
whenever $A$ has positive upper Banach density and
$B$ is infinite (see Proposition \ref{pigeonhole}).
By the same \emph{pigeonhole principle} argument used
in the proof of that proposition, one also shows that:

\smallskip
\begin{itemize}
\item
If $\overline{d}(A)>0$ then $\overline{\Delta}_0(A)$ is syndetic.

\smallskip
\item
If $\text{BD}(A)>0$ then $\Delta_0(A)$ is syndetic.
\end{itemize}

\smallskip
This section aims at sharpening the above results
by considering $\epsilon$-Delta sets (see Definition \ref{def-epsilondeltasets}).
To this end, we shall use the following
combinatorial lemma, which is proved by a straight
application of \emph{Cauchy-Schwartz inequality}.
The main point here is that this result holds in the
nonstandard setting of \emph{hyperintegers}.

\medskip
\begin{lemma}\label{lemmazero}
Let $N\in\hN$ be an infinite hypernatural number,
let $\{C_i\mid i\in\Lambda\}$ be a family of internal
subsets of $[1,N]$, and assume that every
standard part $\text{st}({|C_i|}/{N})\ge\gamma$,
where $\gamma$ is a fixed positive real number.
Then for every $0\le\epsilon<\gamma^2$ and
for every $F\subseteq\Lambda$
with $|F|>\frac{\gamma-\epsilon}{\gamma^2-\epsilon}$,
there exist distinct elements $i,j\in F$ such that
$\text{st}({|C_i\cap C_j|}/{N})>\epsilon$.
\end{lemma}

\begin{proof}
Assume by contradiction that there exists a finite subset $F\subseteq I$
with cardinality $k=|F|>\frac{\gamma-\epsilon}{\gamma^2-\epsilon}$
and such that
$\text{st}({|C_i\cap C_j|}/{N})\le\epsilon$ for all distinct $i,j\in F$.
To simplify matters, let us assume without loss of generality
that all standard parts $\text{st}(|C_i|/N)=\gamma$.
By our hypotheses, we have the following:

\medskip
\begin{itemize}
\item
$c_i=|C_i|/N=\gamma+\eta_i$ where $\eta_i\approx 0$.

\medskip
\item
$\sum_{i\in F}c_i=k\,\gamma+\eta$ where $\eta=\sum_{i\in F}\eta_i\approx 0$.

\medskip
\item
$c_{ij}={|C_i\cap C_j|}/{N}\le\epsilon+\upsilon_{ij}$ where
$\upsilon_{ij}\approx 0$.

\medskip
\item
$\sum_{i\ne j} c_{ij}\le\binom{k}{2}\cdot\epsilon+\upsilon$
where $\upsilon=\sum_{i\ne j}\upsilon_{ij}\approx 0$.
\end{itemize}

\medskip
Now let us denote by $\chi_i:[1,N]\to\{0,1\}$ the
characteristic function of $C_i$. Clearly $c_i=(1/N)\cdot\sum_{\xi=1}^N\chi_i(\xi)$
and $c_{ij}=(1/N)\cdot\sum_{\xi=1}^N\chi_i(\xi)\,\chi_j(\xi)$.
By the \emph{Cauchy-Schwartz inequality}, we obtain:
\begin{eqnarray}
\nonumber
k^2\gamma^2 & \approx & (k\,\gamma + \eta)^2\ =\ \left(\sum_{i\in F} c_i\right)^2\ =\
\frac{1}{N^2}\cdot\left(\,\sum_{i\in F}\left(\sum_{\xi=1}^N \chi_i(\xi)\right)\right)^2
\\
\nonumber
{} & \hspace{-2.5cm} = &
\hspace{-1.3cm}
\frac{1}{N^2}\cdot\left(\,\sum_{\xi=1}^N\ 1\cdot\left(\sum_{i\in F}\chi_i(\xi)\right)\right)^2\ \le\
\frac{1}{N^2}\cdot\left(\sum_{\xi=1}^N\ 1^2\right)\,\cdot\,\sum_{\xi=1}^N\left(\sum_{i\in F}\chi_i(\xi)\right)^2
\\
\nonumber
{} & \hspace{-2.5cm} = &
\hspace{-1.3cm}
\frac{1}{N}\cdot\sum_{\xi=1}^N\left(\sum_{i,j\in F}\chi_i(\xi)\cdot\chi_j(\xi)\right)\ =\
\sum_{i,j\in F}\left(\frac{1}{N}\cdot\sum_{\xi=1}^N\chi_i(\xi)\cdot\chi_j(\xi)\right)
\\
\nonumber
{} & \hspace{-2.5cm} = &
\hspace{-1.3cm}
\sum_{i\in F} c_i\,+\,2\cdot\sum_{i<j}c_{ij}\ \le\
k\cdot\gamma+\eta\,+\, 2\,\binom{k}{2}\,\epsilon\,+\,2\,\upsilon
\\
\nonumber
{} & \hspace{-2.5cm} \approx &
\hspace{-1.3cm}
k\,\gamma\,+\, k\,(k-1)\,\epsilon\ =\ k\,(\gamma+(k-1)\,\epsilon)\,,
\end{eqnarray}
and hence $k\,\gamma^2\le\gamma+(k-1)\,\epsilon$.
This contradicts the assumption $k>\frac{\gamma-\epsilon}{\gamma^2-\epsilon}$.
\end{proof}

\medskip
A consequence of the above lemma that is relevant to our purposes,
is the following one.

\medskip
\begin{lemma}\label{lemmaone}
Let $N\in\hN$ be an infinite hypernatural number, let $C\subseteq[1,N]$ be an internal set
with $\text{st}({|C|}/{N})=\gamma>0$, let $0\le\epsilon<\gamma^2$ be a real number,
and let $k=\lfloor\frac{\gamma-\epsilon}{\gamma^2-\epsilon}\rfloor$. Then
for every infinite set $X\subseteq\Z$ and for every $x\in X$,
there exists a finite subset $F\subset X$ with $x\in F$,
$|F|\le k$, and such that $X\subseteq\D_\epsilon(C)+F$, where
$$\D_\epsilon(C)\ =\
\left\{t\in\Z\,\Big|\,\text{st}\left(\frac{|C\cap(C-t)|}{N}\right)>\epsilon\right\}.$$
\end{lemma}

\begin{proof}
We proceed by induction, and define the finite subset
$F=\{x_i\}_{i=1}^m\subset X$ as follows.
Put $x_1=x$. If $X\subseteq\D_\epsilon(C)+x_1$ then
put $F=\{x_1\}$ and stop. Otherwise pick $x_2\in X$ such that
$x_2\notin\D_\epsilon(C)+x_1$. Then $x_2-x_1$
does not belong to $\D_\epsilon(C)$. So,
$\text{st}(|C\cap(C-x_2+x_1)|/N)\le\epsilon$, and hence also
$\text{st}(|(C-x_1)\cap(C-x_2)|/N)\le\epsilon$,
because $x_1/N\approx 0$. Next, if
$X\subseteq\bigcup_{i=1}^2\left(\D_\epsilon(C)+x_i\right)$,
put $F=\{x_1,x_2\}$ and stop. Otherwise
pick a witness $x_3\in X$ such that $x_3\notin\bigcup_{i=1}^2\D_\epsilon(C)+x_i$.
Then $\text{st}(|C\cap(C-x_3+x_i)|/N)\le\epsilon$ for $i=1,2$, and so also
$\text{st}(|(C-x_i)\cap(C-x_3)|/N)\le\epsilon$, because $x_i/N\approx 0$.
We iterate this process. We now show that the procedure must stop \emph{before} step $k+1$.
If not, one could consider the family $\{C_i\mid i\in[1,k+1]\}$
where $C_i=(C-x_i)\cap[1,N]$. Clearly, $\text{st}(|C_i|/N)=\text{st}(|C|/N)=\gamma$
for all $i$, and by the previous lemma one would have
$\text{st}(|(C-x_i)\cap(C-x_j)|/N)>\epsilon$ for suitable $i,j\in[1,k+1]$,
a contradiction. We conclude that
the cardinality of $F=\{x_i\}_{i=1}^m$ has the desired bound and
$X\subseteq\D_\epsilon(C)+F$.
\end{proof}

\medskip
We now use the above \emph{nonstandard} properties to prove
a general result for sets of positive density.

\medskip
\begin{theorem}\label{theoremone}
Let $\text{BD}(A)=\alpha>0$ (or $\overline{d}(A)=\alpha>0$), and
let $0\le\epsilon<\alpha^2$. Then for every infinite $X\subseteq\Z$
and for every $x\in X$ there exists a finite subset $F\subset X$ such that:

\smallskip
\begin{enumerate}
\item
$x\in F$;

\smallskip
\item
$|F|\le\lfloor\frac{\alpha-\epsilon}{\alpha^2-\epsilon}\rfloor$;

\smallskip
\item
$X\subseteq\Delta_\epsilon(A)+F$
(or $X\subseteq\overline{\Delta}_\epsilon(A)+F$, respectively).
\end{enumerate}
\end{theorem}

\begin{proof}
By the hypothesis $\text{BD}(A)=\alpha$, there exists an infinite hypernatural
number $N\in\hN$ and an hyperinteger $\Omega\in\hZ$ such that
$$\frac{|\hA\cap [\Omega+1,\Omega+N]|}{N}\ \approx\ \alpha.$$
Then $C=(\hA-\Omega)\cap[1,N]$ is an internal
subset of $[1,N]$ with $\text{st}(|C|/N)=\alpha>0$.
By Lemma \ref{lemmaone}, there exists a finite set $F\subset X$
with $x\in F$, $|F|\le\lfloor{(\alpha-\epsilon)}/{(\alpha^2-\epsilon)}\rfloor$
and such that $X\subseteq\D_\epsilon(C)+F$.
To reach the thesis, it is now enough to show that
$\D_\epsilon(C)\subseteq\Delta_\epsilon(A)$.
To see this, take an arbitrary $t\in\D_\epsilon(C)$. Then
\begin{eqnarray}
\nonumber
\text{BD}(A\cap(A-t)) & \ge &
\text{st}\left(\frac{|{}^*(A\cap(A-t))\cap[\Omega+1,\Omega+N]|}{N}\right)\ =\
\\
\nonumber
{} & = & \text{st}\left(\frac{|C\cap (C-t)|}{N}\right)\ >\ \epsilon\,.
\end{eqnarray}

Under the assumption that the upper asymptotic density $\overline{d}(A)=\alpha>0$,
one applies the same argument as above where $\Omega=0$,
and obtains $\D_\epsilon(C)\subseteq\overline{\Delta}_\epsilon(A)$.
\end{proof}

\medskip
As the particular case when $X=\Z$ and $\epsilon=0$,
the above theorem gives a small improvement of a result by I.Z.~Ruzsa
(\emph{cf.} \cite{ru2} Theorem 2), which was a refinement of
a previous result by C.L.~Stewart and R.~Tijdeman \cite{st1}.\footnote
{~The improvement here is that under the hypothesis $\text{BD}(A)=\alpha>0$,
in \cite{ru2} it is proved the weaker property that $\lfloor 1/\alpha\rfloor$-many
shifts of $\{t\in\Z\mid |A\cap(A+t)|=\infty\}$ cover $\Z$.}

\smallskip
For $h\in\N$, denote by

\smallskip
\begin{itemize}
\item
$h\,B=\{h\,b\mid b\in B\}$ the set of $h$-multiples
of elements of $B$\,;

\smallskip
\item
$B/h=\{x\mid h\,x\in B\}$ the set of integers whose $h$-multiples belong to $B$.
\end{itemize}

\smallskip
By taking $X=h\,\Z$ as the set of multiples of a number $h$, one gets the following.

\medskip
\begin{corollary}
Let $\text{BD}(A)=\alpha>0$ (or $\overline{d}(A)=\alpha>0$),
let $0\le\epsilon<\alpha^2$, and let
$k=\lfloor\frac{\alpha-\epsilon}{\alpha^2-\epsilon}\rfloor$.
Then for every $h\in\Z$ there exists a finite set $|F|\le k$
such that $\Z=\Delta_\epsilon(A)/h+F$
(or $\Z=\overline{\Delta}_\epsilon(A)/h+F$, respectively).
In consequence, $\Delta_\epsilon(A)/h$ is syndetic and
$\underline{\text{BD}}\,(\Delta_\epsilon(A)/h)\ge 1/k$
(or $\overline{\Delta}_\epsilon(A)/h$ is syndetic
and $\underline{\text{BD}}\,(\overline{\Delta}_\epsilon(A)/h)\ge 1/k$, respectively).
\end{corollary}

\begin{proof}
Assume first that $\text{BD}(A)=\alpha>0$.
By applying the above theorem with $X=h\,\Z$,
one obtains the existence of a finite set $h\,F\subset h\,\Z$
with $|h\,F|=|F|\le k$ and such that
$h\,\Z\subseteq\Delta_\epsilon(A)+h\,F$.\footnote
{~We assumed $h\ne 0$. Notice that if $h=0$, then trivially $\Delta_\epsilon(A)/h=\Z$
because $0\in\Delta_\epsilon(A)$.}  But then
it follows that $\Z=\Delta_\epsilon(A)/h+F$, and thus $\Delta_\epsilon(A)/h$ is syndetic.
Finally, the last property in the statement follows
by Proposition \ref{lowerbanach}. The second part
of the proof where one assumes $\overline{d}(A)=\alpha>0$
is entirely similar.
\end{proof}

\smallskip
There are potentially many examples to illustrate consequences
of Theorem \ref{theoremone}. For instance, assume that a set $A$ has
Banach density $\text{BD}(A)=\alpha=1/2+\delta$ for $\delta>0$. Then we can
conclude that $\text{BD}(A\cap(A-t))\ge\delta+2\delta^2$
for all $t\in\Z$. In fact, given any $\epsilon<\delta+2\,\delta^2$, we have that
${(\alpha-\epsilon)}/{(\alpha^2-\epsilon)}<2$ and so, by
taking $X=\Z$, it follows that $\Delta_\epsilon(A)=\Z$.
It seems worth investigating the possibility of deriving other
consequences of Theorem \ref{theoremone},
by means of suitable choices of the set $X$.

\bigskip
\section{Finite embeddability}\label{sec-embeddability}

As already remarked in Section \ref{sec-preliminaries},
the finite embeddability relation  (see Definition \ref{def-finiteembeddability})
preserves the finite combinatorial structure
of sets, including many familiar notions considered in combinatorics
of integer numbers. A first list is given below.
(All proofs follow from the definitions in
a straightforward manner, and are omitted.)

\medskip
\begin{proposition}
\

\begin{enumerate}
\item
A set is $\lhd$-maximal if and only if it is $\lhd_d$-maximal
if and only if it is thick.

\smallskip
\item
If $X\lhd Y$ and $X$ is piecewise syndetic,
then also $Y$ is piecewise syndetic.

\smallskip
\item
If $X\lhd Y$ and $X$ contains an arithmetic progression of length $k$, then
also $Y$ contains an arithmetic progression of length $k$.

\smallskip
\item
If $X\lhd_d Y$ and if $X$ contains an arithmetic progression of
length $k$ and common distance $d$, then
$Y$ contains ``densely-many'' such arithmetic progressions,
\emph{i.e.}
$\text{BD}\left(\{x\in\Z\mid x, x+d, \ldots, x+(k-1)d\in Y\}\right)>0$.

\smallskip
\item
If $X\lhd Y$ then $\text{BD}(X)\le\text{BD}(Y)$.
\end{enumerate}
\end{proposition}

\medskip
Remark that while piecewise syndeticity is preserved
under $\lhd$, the property of being syndetic is \emph{not}. Similarly,
the upper Banach density is preserved or increased under $\lhd$,
but upper asymptotic density is \emph{not}.
Another list of basic properties of embeddability that
are relevant to our purposes are itemized below.

\medskip
\begin{proposition}
\

\begin{enumerate}
\item
If $X\lhd Y$ and $Y\lhd Z$ then $X\lhd Z$.

\smallskip
\item
If $X\lhd Y$ and $Y\lhd_d Z$ then $X\lhd_d Z$.

\smallskip
\item
If $X\lhd_d Y$ and $Y\lhd Z$ then $X\lhd_d Z$.

\smallskip
\item
If $X\lhd Y$ then $\Delta(X)\subseteq\Delta(Y)$.

\smallskip
\item
If $X\lhd_d Y$ then $\Delta(X)\subseteq\Delta_0(Y)$.

\smallskip
\item
If $X\lhd Y$ and $X'\lhd Y'$ then $X-X'\lhd Y-Y'$.

\smallskip
\item
If $X\lhd_d Y$ and $X'\lhd Y'$ then $X-X'\lhd_d Y-Y'$.

\smallskip
\item
If $X\lhd Y$ then $\bigcap_{t\in G}(X-t)\lhd \bigcap_{t\in G}(Y-t)$ for every finite $G$.

\smallskip
\item
If $X\lhd_d Y$ then $\bigcap_{t\in G}(X-t)\lhd_d\bigcap_{t\in G}(Y-t)$ for every finite $G$.

\smallskip
\item
If $X\lhd Y$ then
$\Delta_\epsilon(X)\subseteq\Delta_\epsilon(Y)$ for all $\epsilon\ge 0$.
\end{enumerate}
\end{proposition}

\begin{proof}
$(1)$ is straightforward from the definition of $\lhd$.

\smallskip
$(2)$. Given a finite $F\subseteq X$, pick $t$ such that
$t+F\subseteq Y$. As the Banach density is shift invariant, we have:
$$\text{BD}\left(\bigcap_{x\in F}Z-x\right)\ =\
\text{BD}\left(\bigcap_{x\in F}Z-x-t\right)\ =\
\text{BD}\left(\bigcap_{s\in t+F}Z-s\right)\ >\ 0.$$

\smallskip
$(3)$. Given a finite $F\subseteq X$, denote by
$A=\bigcap_{x\in F}(Y-x)$ and by $B=\bigcap_{x\in F}(Z-x)$.
By the hypothesis $X\lhd_d Y$, we know that
$\text{BD}(A)>0$. If we show that $A\lhd B$, then
the thesis will follow from item $(5)$ of the previous proposition.
Let $G\subseteq A$ be finite; then for all $x\in F$ and for all $\xi\in G$,
we have $x+\xi\in Y$, \emph{i.e.} $F+G\subseteq Y$. By the hypothesis $Y\lhd Z$,
we can pick $t$ such that $t+F+G\subseteq Z$, and hence
$t+G\subseteq\bigcap_{x\in F}(Z-x)$, as desired.

\smallskip
$(4)$. Given $x,x'\in X$, by the hypothesis
we can pick $t$ such that $t+\{x,x'\}\subseteq Y$.
But then $x-x'=(t+x)-(t+x')\in\Delta(Y)$.

\smallskip
$(5)$. For $x,x'\in X$, we have that $\text{BD}(Y\cap(Y-x+x'))=
\text{BD}((Y-x')\cap(Y-x))>0$, and so $x-x'\in\Delta_0(Y)$.

\smallskip
$(6)$. Given a finite $F\subseteq X-X'$, let $G\subseteq X$ and $G'\subseteq X'$
be finite sets such that $F\subseteq G-G'$. By the hypotheses,
there exist $t,t'$ such that $t+G\subseteq Y$ and
$t'+G'\subseteq Y'$. Then
$(t-t')+F\subseteq(t+G)-(t'+G')\subseteq Y-Y'$.

\smallskip
$(7)$. As above, given a finite $F\subseteq X-X'$, pick finite $G\subseteq X$
and $G'\subseteq X'$ such that $F\subseteq G-G'$.
By the hypothesis $X\lhd_d Y$,
the set $\Gamma=\{t\mid t+G\subseteq Y\}$ has positive
upper Banach density; and by the hypothesis $X'\lhd Y'$,
there exists an element $s$ such that $s+G'\subseteq Y'$.
For all $t\in\Gamma$, we have that $t-s+F\subseteq t-s+(G-G')=(t+G)-(t'+G')\subseteq Y-Y'$.
This shows that $\Gamma-s\subseteq \{w\mid w+F\subseteq Y-Y'\}$, and we
conclude that also the latter set has positive
upper Banach density, as desired.

\smallskip
$(8)$. Let a finite set $F\subseteq\bigcap_{t\in G}(X-t)$ be given.
Notice that $F+G\subseteq X$, so we can pick an element $w$ such that
$w+(F+G)\subseteq Y$. Then $w+F\subseteq\bigcap_{t\in G}Y-t$.

\smallskip
$(9)$. Proceed as above, by noticing that the set
$\{w\mid w+F\subseteq\bigcap_{t\in G}Y-t\}$
has positive Banach density in that a superset of
$\{w\mid w+F+G\subseteq Y\}$.

\smallskip
$(10)$. By property $(8)$, it follows that
$(X\cap(X-t))\lhd(Y\cap(Y-t))$ for every $t$.
This implies that $\text{BD}(X\cap(X-t))\le\text{BD}(Y\cap(Y-t))$, and
the desired inclusion follows.
\end{proof}

\medskip
In a nonstandard setting, the finite embeddability $X\lhd Y$
amounts to the property that a (possibly infinite) shift of $X$ is
included in the hyper-extension $\hY$. This notion can also be
characterized in terms of \emph{ultrafilter-shifts}, as defined
by M.~Beiglb\"ock in \cite{bei}.

\medskip
\begin{proposition}\label{fe}
Let $X,Y\subseteq\Z$. Then the following are equivalent:

\begin{enumerate}
\item
$X\lhd Y$.

\smallskip
\item
$\mu+X\subseteq\hY$ for some $\mu\in\hZ$.

\smallskip
\item
There exists an ultrafilter $\U$ on $\Z$ such that
$X$ is a subset of the ``$\U$-shift'' of $Y$, namely
$Y-\U=\{t\in\Z\mid Y-t\in\U\}\supseteq X$.
\end{enumerate}
\end{proposition}

\begin{proof}
$(1)\Rightarrow(2)$. Let $X=\{x_n\mid n\in\N\}$.
By the hypothesis $X\lhd Y$, for every $n\in\N$, the finite intersection
$\bigcap_{i=1}^n(Y-x_i)\ne\emptyset$. Then, by \emph{overspill},
there exists an infinite $N\in\hN$ such that
$\bigcap_{i=1}^N(\hY-x_i)$ is non-empty.
If $\mu\in\hZ$ is any hyperinteger in that intersection,
then clearly $\mu+x_i\in\hY$ for all $i\in\N$.

$(2)\Rightarrow(3)$.
Let $\U=\{A\subseteq\Z\mid \mu\in\hA\}$. It is readily verified
that $\U$ is actually an ultrafilter on $\Z$. For every $x\in X$,
by the hypothesis $\mu+x\in\hY\Rightarrow\mu\in{}^*(Y-x)$,
and hence $Y-x\in\U$, \emph{i.e.} $x\in Y-\U$, as desired.

$(3)\Rightarrow(1)$. Given a finite $F\subset X$,
the set $\bigcap_{x\in F}(Y-x)$ is nonempty, because
it is a finite intersection of elements of $\U$.
If $t\in\Z$ is any element in that intersection, then
$t+F\subset Y$.
\end{proof}

\medskip
An interesting nonstandard property discovered by
R.~Jin \cite{jin2} is the fact that whenever an internal set
of hypernatural numbers $C\subseteq[1,N]\subset\hN$
has a non-infinitesimal relative density $\text{st}({|C|}/{N})=\gamma>0$, then
$C$ must include a translated copy of a set $E\subseteq\N$
whose \emph{Schnirelmann density} is
at least $\gamma$. Below, we prove the related property
that one can find a set $E\subseteq\N$ with Schnirelmann density
at least $\gamma$, and such that ``many'' translated
copies of its initial segments $E\cap[1,n]$ are exactly found in $C$.\footnote
{~The argument used in this proof is essentially due to C.L.~Stewart
and R.~Tijdeman (see Theorem 1 of \cite{st1}).}

\medskip
\begin{lemma}\label{sigmalemma}
Let $N\in\hN$ be an infinite hypernatural number, and
let $C\subseteq[1,N]$ be an internal set with $\text{st}({|C|}/{N})=\gamma>0$.
Then there exists a set $E\subseteq\N$ of natural numbers such that

\begin{enumerate}
\item
The Schnirelmann density $\sigma(E)\ge\gamma$;

\smallskip
\item
Every internal set
$\Theta_n=\left\{\theta\in[1,N]\mid (C-\theta)\cap[1,n]=E\cap[1,n]\right\}$
is such that $\text{st}({|\Theta_n|}/{N})>0$.
\end{enumerate}
\end{lemma}

\begin{proof}
For every $n\in\N$, let
$$\Gamma_n\ =\ \left\{\theta\in[1,N]\,\Big|\,
\min_{1\le i\le n}\frac{|C\cap[\theta+1,\theta+i]|}{i}\ge\gamma\right\},$$
and let $\Lambda_n=[1,N]\setminus \Gamma_n$ be its complement. Notice that
$$\Lambda_n\ =\ \left\{\theta\in[1,N]\,\Big|\,
\min_{1\le i\le n}\frac{|C\cap[\theta+1,\theta+i]|}{i}\le\gamma_n\right\}\,,$$
where $\gamma_n<\gamma$ is the rational number
$\gamma_n=\max\{\frac{j}{i}<\gamma\mid 1\le i\le n,\ 0\le j\le i\}$.
Define the internal map $F$ on $[1,N]$ by putting:
$$F(\theta)\ =\
\begin{cases}
1 & \text{if}\ \theta\in\Gamma_n
\\
s & \text{if } \theta\in\Lambda_n\ \text{and }
s=\min\left\{1\le i\le n\,\big|\,\frac{|C\cap[\theta+1,\theta+i]|}{i}\le\gamma_n\right\}.
\end{cases}$$

By internal induction, define a hyperfinite sequence
by putting $\theta_0=1$, and $\theta_{m+1}=\theta_m+F(\theta_m)$
as long as $\theta_{l+1}\le N+1$. Notice that, since $F(\theta)\le n$ for all $\theta$,
the set $[1,N]\setminus[\theta_0,\theta_{l+1})$ contains
less than $n$-many elements. Then we have:

\begin{eqnarray}
\nonumber
|C| & < &
|C\cap[\theta_0,\theta_{l+1})|+n\ =\
\sum_{i=0}^l|C\cap[\theta_i,\theta_{i+1})|+n
\\
\nonumber
{}  & = &
\sum_{\substack{0\le i\le l \\ \theta_i\in\Gamma_n}}|C\cap[\theta_i,\theta_{i+1})|\ +\
\sum_{\substack{0\le i\le l \\ \theta_i\in\Lambda_n}}|C\cap[\theta_i,\theta_{i+1})|\ + n
\\
\nonumber
{} & \le &
\sum_{\substack{0\le i\le l \\ \theta_i\in\Gamma_n}}|C\cap\{\theta_i\}|\ +\
\sum_{\substack{0\le i\le l \\ \theta_i\in\Lambda_n}}|C\cap[\theta_i,\theta_{i+1})|\ + n.
\end{eqnarray}

Now, let $X=\{\theta_i\mid i=0,\ldots,l\}$.
In the last line above, the first term equals
$|C\cap X\cap\Gamma_n|\le |X\cap\Gamma_n|$, and the second term:
\begin{eqnarray}
\nonumber
\sum_{\substack{0\le i\le l \\ \theta_i\in\Lambda_n}}|C\cap[\theta_i,\theta_{i+1})| & \le &
\sum_{\substack{0\le i\le l \\ \theta_i\in\Lambda_n}}F(\theta_i)\cdot\gamma_n
\\
\nonumber
{}  & = &
\gamma_n\cdot\left(\sum_{0\le i\le l}F(\theta_i) -
\sum_{\substack{0\le i\le l \\ \theta_i\in\Gamma_n}} 1\right)
\\
\nonumber
{} & = &
\gamma_n\cdot\left(\theta_{l+1}-1-|X\cap\Gamma_n|\right)\ \le\
\gamma_n\cdot(N-|X\cap\Gamma_n|).
\end{eqnarray}

So, we have the inequality $|C|<M_n+\gamma_n(N-M_n)+n$ where
$M_n=|X\cap\Gamma_n|$, and we obtain that:
$$\frac{|\Gamma_n|}{N}\ \ge\ \frac{M_n}{N}\ >\
\frac{|C|/N - \gamma_n - n/N}{1-\gamma_n}.$$

Notice that the last quantity has a positive standard part.
As there are $2^n$-many subsets of $[1,n]$, by the
\emph{pigeonhole principle} there exists a subset $\Gamma'_n\subseteq\Gamma_n$
with $|\Gamma'_n|\ge|\Gamma_n|/{2^n}$ and a set $B_n\subseteq[1,n]$
with the property that $(C-\theta)\cap[1,n]=B_n$
for all $\theta\in\Gamma'_n$.

\smallskip
Now fix a non-principal
ultrafilter $\U$ on $\N$, and define the set $E\subseteq\N$ by putting
$$n\in E\ \Leftrightarrow\ \B_n=\{k\ge n\mid n\in B_k\}\in\U.$$

We claim that $E$ is the desired set.
Given $n$, the following set belongs to $\U$, because it is
a finite intersection of elements of $\U$:
$$\bigcap_{i\in E\cap[1,n]}\!\!\!\!\!\B_i\ \ \ \cap\
\bigcap_{i\in[1,n]\setminus E}\!\!\!\!\!\B_i^c\ \in\ \U.$$
(Notice that, since $\gamma>0$, we have $1\in B_k$ for all $k$,
and so $1\in E\cap[1,n]\ne\emptyset$.)
If $k$ is any number in the above intersection,
then $B_k\cap[1,n]=E\cap[1,n]$. Moreover, for every $\theta\in\Gamma'_k$,
$$\frac{|E\cap[1,n]|}{n}\ =\ \frac{|B_k\cap[1,n]|}{n}\ \ge\
\min_{1\le i\le k}\frac{|B_k\cap[1,i]|}{i}\ =\
\min_{1\le i\le k}\frac{|C\cap[\theta+1,\theta+i]|}{i}\ \ge\ \gamma.$$
This proves that $\sigma(E)\ge\gamma$. Moreover,
$\theta\in\Gamma'_k\Rightarrow (C-\theta)\cap[1,k]=B_k\Rightarrow
(C-\theta)\cap[1,n]=E\cap[1,n]$, and hence $\theta\in\Theta_n$.
Therefore we conclude that
$$\frac{|\Theta_n|}{N}\ \ge\ \frac{|\Gamma'_k|}{N}\ >\
\frac{|C|/N - \gamma_k - k/N}{2^k(1-\gamma_k)}\,,$$
where the standard part of the last quantity is $\frac{\gamma-\gamma_k}{2^k(1-\gamma_k)}>0$.
\end{proof}

\medskip
In consequence of the previous \emph{nonstandard} lemma,
we obtain an embeddability property that
holds for all sets of positive density.
It is a small refinement of a result by V.~Bergelson \cite{be1}, which
improved on a previous result by C.L.~Stewart and R.~Tijdeman \cite{st2}.\footnote{
~The improvement here is that we have $\sigma(E)\ge\alpha$
instead of $\overline{d}(E)\ge\alpha$.}

\medskip
\begin{theorem}[\emph{Cf.} \cite{be1} Theorem 2.2; \cite{st2} Theorem 1]\label{theoremtwo}
Let $\text{BD}(A)=\alpha>0$. Then there exists
a set of natural numbers $E\subseteq\N$ such that:

\begin{enumerate}
\item
$\sigma(E)\ge\alpha$.

\smallskip
\item
$E\lhd_d A$, and hence $\Delta(E)\subseteq\Delta_0(A)$
and $\Delta_\epsilon(E)\subseteq\Delta_\epsilon(A)$ for all $\epsilon\ge 0$.

\end{enumerate}
\end{theorem}

\begin{proof}
Pick an infinite interval $[\Omega+1,\Omega+N]$ such that
${|\hA\cap[\Omega+1,\Omega+N]|}/N\approx\alpha$. By applying the
above theorem where $C=(\hA-\Omega)\cap[1,N]$, one gets
the existence of a set $E\subseteq\N$ such that
$\sigma(E)\ge\alpha$ and $\text{st}({|\Theta_n|}/N)>0$ for all $n$, where
$$\Theta_n\ =\ \left\{\theta\in[1,N]\mid (\hA-\Omega-\theta)\cap[1,n]=E\cap[1,n]\right\}.$$

Now, given a finite $F\subseteq E\cap[1,n]$ and given
an element $e\in F$, for every $\theta\in\Theta_n$ we have
$\Omega+\theta+e\in\hA$. This shows that
$\Omega+\Theta_n\subseteq\bigcap_{e\in F}{}^*(A-e)\cap[\Omega+1,\Omega+N]$.
But then
\begin{eqnarray}
\nonumber
\text{BD}\left(\bigcap_{e\in F}(A-e)\right) & \ge &
\text{st}\left(\frac{|\bigcap_{e\in F}{}^*(A-e)\cap[\Omega+1,\Omega+N]|}{N}\right)
\\
\nonumber
{} & \ge &
\text{st}\left(\frac{|\Omega+\Theta_n|}{N}\right)\ =\
\text{st}\left(\frac{|\Theta_n|}{N}\right)\ >\ 0.
\end{eqnarray}
\end{proof}

\bigskip
\section{Difference sets $A-B$}

In this final section we generalize the results of Section \ref{sec-densitydelta}
by considering sets of differences $A-B$ where $A\ne B$.
Remark that while $\Delta(A)=A-A$ is syndetic whenever $A$ has a positive upper Banach density,
the same property does not extend to the case of difference sets $A-B$
where $A\ne B$. (\emph{E.g.}, it is not hard
to construct thick sets $A,B,C$ such that
their complements $A^c,B^c,C^c$ are thick as well, and $A-B\subset C$.)

\smallskip
We shall use the following elementary inequality.

\medskip
\begin{lemma}\label{lemmatwo}
Let $C\subseteq[1,N]$ and $D\subseteq[1,\nu]$ be sets of natural numbers.
Then there exists $1\le\overline{x}\le N$ such that
$$\frac{|(C-\overline{x})\cap D|}{\nu}\ \ge\
\frac{|C|}{N}\cdot\frac{|D|}{\nu}\ -\ \frac{|D|}{N}.$$
\end{lemma}

\begin{proof}
Let $\chi_C:[1,N]\to\{0,1\}$ denote the characteristic function of $C$.
For every $d\in D$ we have
$$\frac{1}{N}\cdot\sum_{x=1}^{N}\chi_C(x+d)\ =\
\frac{|C\cap[1+d,N+d]|}{N}\ =\
\frac{|C|}{N}+\frac{e(d)}{N}$$
where $|e(d)|\le d$. Then:

$$\frac{1}{N}\cdot\sum_{x=1}^N\left(\frac{1}{\nu}\cdot\sum_{d\in D}\chi_C(x+d)\right) \ = \
\frac{1}{\nu}\cdot\sum_{d\in D}\left(\frac{1}{N}\cdot\sum_{x=1}^N\chi_C(x+d)\right)\ =$$
$$=\ \frac{1}{\nu}\cdot\sum_{d\in D}\frac{|C|}{N}\ +\
\frac{1}{N\cdot\nu}\cdot\sum_{d\in D}e(d)\ =\ \frac{|C|}{N}\cdot\frac{|D|}{\nu}\ +\ e$$
where
$$|e|\ =\ \left|\frac{1}{N\cdot\nu}\sum_{d\in D}e(d)\right|\ \le\
\frac{1}{N\cdot\nu}\sum_{d\in D}|e(d)|\ \le\ \frac{1}{N\cdot\nu}\cdot\sum_{d\in D}d\ \le\
\frac{1}{N\cdot\nu}\sum_{d\in D}\nu\ =\ \frac{|D|}{N}.$$

By the \emph{pigeonhole principle}, there must exists at least one number $1\le \overline{x}\le N$ such that
$$\frac{1}{\nu}\cdot\sum_{d\in D}\chi_C(\overline{x}+d)\ \ge\
\frac{|C|}{N}\cdot\frac{|D|}{\nu}\ -\ \frac{|D|}{N}.$$
The thesis is reached by noticing that
$$\frac{1}{\nu}\cdot\sum_{d\in D}\chi_C(\overline{x}+d)\ =\
\frac{|(D+\overline{x})\cap C|}{\nu}\ =\ \frac{|(C-\overline{x})\cap D|}{\nu}.$$
\end{proof}

\medskip
We are now ready to prove the main result of this paper.

\medskip
\begin{theorem}\label{theoremthree}
Let $\text{BD}(A)=\alpha>0$ and $\text{BD}(B)=\beta>0$.
Then there exists a set of natural numbers $E\subseteq\N$ such that:

\begin{enumerate}
\item
The Schnirelmann density $\sigma(E)\ge\alpha\beta$.

\smallskip
\item
For every finite $F\subset E$ there exists $\epsilon>0$
such that for arbitrarily large intervals $J$
one finds a suitable shift $A_J=A-t_J$ with the property that
$$\frac{|\left(\bigcap_{e\in F}(A_J\cap B)-e\right)\cap J|}{|J|}\ \ge\ \epsilon.$$

\smallskip
\item
Both $E\lhd_d A$ and $E\lhd_d B$, and hence:

\smallskip
\begin{itemize}
\item
$\Delta(E)\subseteq\Delta_0(A)\cap\Delta_0(B)$;

\smallskip
\item
$\Delta_\epsilon(E)\subseteq\Delta_\epsilon(A)\cap\Delta_\epsilon(B)$ for all $\epsilon\ge 0$;

\smallskip
\item
$\Delta(E)\lhd_d A-B$.
\end{itemize}

\end{enumerate}
\end{theorem}

\begin{proof}
$(1)$. Fix $\nu,N\in\hN$ infinite numbers with $\nu/N\approx 0$, and pick
$[\Omega+1,\Omega+N]$ and $[\Xi+1,\Xi+\nu]$ intervals of length $N$ and $\nu$ respectively,
such that
$$\frac{|\hA\cap[\Omega+1,\Omega+N]}{N}\,\approx\,\alpha\quad\text{and}\quad
\frac{|\hB\cap[\Xi+1,\Xi+\nu]|}{\nu}\,\approx\,\beta.$$

Consider the internal sets
\begin{itemize}
\item
$C=(\hA-\Omega)\cap[1,N]$\,;

\smallskip
\item
$D=(\hB-\Xi)\cap[1,\nu]$.
\end{itemize}

\smallskip
Clearly $|C|/N\approx\alpha$
and $|D|/\nu\approx\beta$.
The property of Lemma \ref{lemmatwo} \emph{transfers} to the
internal sets $C\subseteq[1,N]$ and $D\subseteq[1,\nu]$, and so
we can pick a hyperinteger element $1\le\zeta\le N$ such that
$$\frac{|(C-\zeta)\cap D|}{\nu}\ \ge\
\frac{|C|}{N}\cdot\frac{|D|}{\nu}\ -\ \frac{|D|}{N}.$$

Now let $W=(C-\zeta)\cap D\subseteq[1,\nu]$.
Since $|D|/N\le\nu/N\approx 0$, we have that
$$\gamma\ =\ \text{st}\left(\frac{|W|}{\nu}\right)\ \ge\
\text{st}\left(\frac{|C|}{N}\cdot\frac{|D|}{\nu}\right)\ =\
\text{st}\left(\frac{|C|}{N}\right)\cdot\text{st}\left(\frac{|D|}{\nu}\right)\ =\
\alpha\cdot\beta.$$

By applying Theorem \ref{theoremtwo} to the internal set $W\subseteq[1,\nu]$,
one gets the existence of a set $E\subseteq\N$ satisfying the
following properties:

\smallskip
\begin{itemize}
\item
$\sigma(E)\ge\gamma\ge\alpha\beta$.

\smallskip
\item
For every $n$, the internal set
$\Theta_n=\{\theta\in[1,\nu]\mid (W-\theta)\cap[1,n]=E\cap[1,n]\}$
is such that $\text{st}({|\Theta_n|}/{\nu})>0$.
\end{itemize}

$(2)$.
Given a finite set $F=\{e_1<\ldots<e_k\}\subseteq E\cap[1,n]$,
for every $\theta\in\Theta_n$ and for every $i$ we have that
$\theta+e_i\in W=(\hA-\Omega-\zeta)\cap(\hB-\Xi)\cap[1,\nu]$, and so
$$\Xi+\Theta_n\ \subseteq\ \bigcap_{i=1}^k[\left((\hA-\mu)\cap \hB\right)-e_i]\cap I$$
where $\mu=\Omega+\zeta-\Xi$ and $I=[\Xi+1,\Xi+\nu]$. Then
$$\hspace{-2cm}(\star)\quad\quad
\text{st}\left(\frac{|\bigcap_{i=1}^k[\left((\hA-\mu)\cap \hB\right)-e_i]\cap I|}
{|I|}\right)\ \ge\ \text{st}({|\Theta_n|}/{\nu})\ =\ \epsilon\ >\ 0.$$

We now want to extract a standard property out of the above
nonstandard inequality $(\star)$.
Notice that, since $|I|=\nu$ is infinite,
the following is true for every fixed $m\in\N$:
$$\exists I\subset\hZ\ \text{interval s.t. } |I|>m\ \&\
\exists \mu\in\hZ\ \ \text{s.t.}\ \frac{|\bigcap_{i=1}^k
[\left((\hA-\mu)\cap \hB\right)-e_i]\cap I|}{|I|}\ \ge\ \epsilon.$$
By \emph{transfer}, we obtain the existence of
an interval $J\subset\Z$ of length $|J|>m$ and of an element $t_J\in\Z$
with the desired property that:
$$\frac{|\bigcap_{i=1}^k[\left((A-t_J)\cap B\right)-e_i]\cap J|}{|J|}\ \ge\ \epsilon.$$

\smallskip
$(3)$. With the same notation as above, by $(\star)$ one directly gets that
$$\text{st}\left(\frac{|\bigcap_{i=1}^k(\hA-e_i)\cap I'|}
{|I'|}\right)\ \ge\ \epsilon\quad\text{and}\quad
\text{st}\left(\frac{|\bigcap_{i=1}^k(\hB-e_i)\cap I|}
{|I|}\right)\ \ge \epsilon,$$
where $I'=\mu+I=[\Omega+\zeta+1,\Omega+\zeta+\nu]$.
Since ${}^*\!\left(\bigcap_{i=1}^k(A-e_i)\right)=\bigcap_{i=1}^k(\hA-e_i)$,
by the nonstandard characterization of upper Banach density, we obtain
the thesis $\text{BD}\left(\bigcap_{e\in F}(A-e)\right)>0$.
The other inequality $\text{BD}\left(\bigcap_{e\in F}(B-e)\right)>0$
is proved in the same way.
\end{proof}

\medskip
By a recent result obtained by M.~Beiglb\"ock, V.~Bergelson and A.~Fish
in the general context of countable amenable groups (\emph{cf.} \cite{bbf} Proposition 4.1.)
one gets the existence of a set $E$ of positive upper Banach density
with the property that $\Delta(E)\lhd A-B$. Afterwards,
M.~Beiglb\"ock found a short ultrafilter proof of that property,
with the refinement that one can take $\text{BD}(E)\ge\alpha\beta$.
Our improvement here is that one can also assume the Schnirelmann
density $\sigma(E)\ge\alpha\beta$, and that there are
\emph{dense} embeddings $E\lhd_d A$ and $E\lhd_d B$
(and hence, a dense embedding $\Delta(E)\lhd_d A-B$).

\smallskip
As a first corollary to our main theorem, we obtain a sharpening
of a result by I.Z.~Rusza \cite{ru2}
about intersections of difference sets, which improved
on a previous result by C.L.~Stewart and R.~Tijedman \cite{st1}.

\medskip
\begin{corollary}[\emph{cf.} \cite{ru2} Theorem 1; \cite{st2} Theorem 4]
Assume that $A_1,\ldots,A_n\subseteq\Z$ have
positive upper Banach densities $\text{BD}(A_i)=\alpha_i$.
Then there exists a set $E\subseteq\N$ with
$\sigma(E)\ge\prod_{i=1}^n\alpha_i$ and such that
$\Delta_\epsilon(E)\subseteq\bigcap_{i=1}^n\Delta_\epsilon(A_i)$
for every $\epsilon\ge 0$.
\end{corollary}

\begin{proof}
We proceed by induction on $n$. The basis $n=1$ is given by
Theorem \ref{theoremtwo}.
At step $n+1$, by the inductive hypothesis we can pick
a set $E'\subseteq\N$ such that $\sigma(E')\ge\prod_{i=1}^n\alpha_i$
and $\Delta_\epsilon(E')\subseteq\bigcap_{i=1}^n\Delta_\epsilon(A_i)$.
Now apply the above theorem to the sets $E'$ and $A_{n+1}$, and obtain
the existence of a set $E\subseteq\N$ whose Schnirelmann density
$\sigma(E)\ge\text{BD}(E')\cdot\text{BD}(A_{n+1})\ge\prod_{i=1}^{n+1}\alpha_i$,
and such that
$\Delta_\epsilon(E)\subseteq\Delta_\epsilon(E')\cap\Delta_\epsilon(A_{n+1})\subseteq
\bigcap_{i=1}^{n+1}\Delta_\epsilon(A_i)$, as desired.
\end{proof}

\medskip
Two more corollaries are obtained by combining Theorem \ref{theoremthree}
with Theorem \ref{theoremone}.

\medskip
\begin{corollary}
Assume that $\text{BD}(A)=\alpha>0$ and $\text{BD}(B)=\beta>0$. Then
for every $0\le\epsilon<\alpha^2\beta^2$,
for every infinite $X\subseteq\Z$, and for every $x\in X$,
there exists a finite subset $F\subset X$ such that

\begin{enumerate}
\item
$x\in F$\,;

\smallskip
\item
$|F|\le\lfloor\frac{\alpha\beta-\epsilon}{\alpha^2\beta^2-\epsilon}\rfloor=k$\,;

\smallskip
\item
$X\subseteq (\Delta_\epsilon(A)\cap\Delta_\epsilon(B))+F$.
\end{enumerate}
In consequence, for every $h$ the intersection $\Delta_\epsilon(A)/h\cap\Delta_\epsilon(B)/h$ is
syndetic and its lower Banach density is not smaller than $1/k$.
\end{corollary}

\begin{proof}
Pick a set $E\subseteq\N$ as given by Theorem \ref{theoremthree}.
As $\overline{d}(E)\ge\sigma(E)\ge\alpha\beta$, we can apply Theorem
\ref{theoremone} and obtain the existence of a finite $F\subset X$ such that
$x\in F$, $|F|\le k$,
and $X\subseteq\Delta_\epsilon(E)+F$
(in fact, $X\subseteq\overline{\Delta}_\epsilon(E)+F$).
As $E\lhd A$ and $E\lhd B$ (in fact, $E\lhd_d A$ and $E\lhd_d B$), we have the
inclusion $\Delta_\epsilon(E)\subseteq \Delta_\epsilon(A)\cap\Delta_\epsilon(B)$, and
the thesis follows. Finally, by taking as $X=h\,\Z$ the set of $h$-multiples,
one obtains that $\Z=(\Delta_\epsilon(A)/h\cap\Delta_\epsilon(B)/h)+G$ for
a suitable $|G|\le k$ and so, by Proposition \ref{lowerbanach},
also the last statement in the corollary is proved.
\end{proof}

\medskip
\begin{corollary}
Assume that $\text{BD}(A)=\alpha>0$ and $\text{BD}(B)=\beta>0$.
Then for every infinite $X\subseteq\Z$ and
for every $x\in X$, there exists a finite subset $F\subset X$ such that

\begin{enumerate}
\item
$x\in F$\,;

\smallskip
\item
$|F|\le\lfloor\frac{1}{\alpha\beta}\rfloor$\,;

\item
$X\lhd_d (A-B)+F$.
\end{enumerate}
\end{corollary}

\begin{proof}
With the same notation as in the proof of the previous corollary,
let $\epsilon=0$ and pick $E\subseteq\N$ with $\sigma(E)\ge\alpha\beta$,
and $|F|\le\lfloor 1/{\alpha\beta}\rfloor$
such that $X\subseteq\Delta_0(E)+F$.
Now, $E\lhd_d A$ and $E\lhd_d B$ imply that
$\Delta(E)\lhd_d A-B$, which in turn implies that $\Delta(E)+F\lhd_d (A-B)+F$.
As $\Delta_0(E)\subseteq\Delta(E)$, we  can conclude that
$X\subseteq \Delta(E)+F\lhd_d (A-B)+F$.
\end{proof}

\medskip
By the above result where $X=\Z$,
one can improve Jin's theorem about the piecewise syndeticity of
a difference set, by giving a precise bound on
the number of shifts of $A-B$ which are
needed to cover a thick set.

\medskip
\begin{corollary}[\emph{cf.} \cite{jin1} Corollary 3]
Assume that $\text{BD}(A)=\alpha>0$ and $\text{BD}(B)=\beta>0$, and
let $k=\lfloor 1/{\alpha\beta}\rfloor$. Then there exists a finite set
$|F|\le k$ such that
$A-B+F$ is thick, and hence $A-B$ is piecewise syndetic.
\end{corollary}

\begin{proof}
Apply the above Corollary with $X=\Z$, and recall that
$\Z\lhd_d Y$ if and only if $Y$ is $\lhd_d$-maximal
if and only if $Y$ is thick.
\end{proof}

\medskip
Bohr sets are commonly used in applications of \emph{Fourier analysis}
in combinatorial number theory. Recall that $A\subseteq\Z$
is called a \emph{Bohr set} if it contains a non-empty open set of
the topology induced by the embedding into the \emph{Bohr compactification}
of the discrete topological group $(\Z,+)$.
The following characterization holds:
$A$ is a Bohr set if and only if there exist
$r_1,\ldots,r_k\in[0,1)$ and a positive $\epsilon>0$
such that a shift of $\{x\in\Z\mid \|r_1\cdot x\|,\ldots,\|r_k\cdot x\|<\epsilon\}$
is included in $A$, where we denoted by $\|z\|$ the distance of
$z$ from the nearest integer.
A set is \emph{piecewise Bohr} if it is the intersection of a Bohr set
with a thick set. We remark that Bohr sets are syndetic,
and hence piecewise Bohr sets are piecewise syndetic, but there are
syndetic sets which are not piecewise Bohr.
(For a proof of this fact, and for more information about
Bohr sets, we refer the reader to \cite{bfw} and references therein.)

\smallskip
As a consequence of Theorem \ref{theoremthree},
one can also recover the following theorem by V.~Bergelson, H.~F\"urstenberg
and B.~Weiss' about the Bohr property of difference sets.

\medskip
\begin{corollary}[\emph{cf.} \cite{bfw} Theorem I]
Let $A$ and $B$ have positive upper Banach
density. Then the difference set $A-B$ is piecewise Bohr.
\end{corollary}

\begin{proof}
By Theorem \ref{theoremthree}, we can pick a set $E\subseteq\N$
with $\sigma(E)\ge\alpha\beta>0$ and such that $\Delta(E)\lhd_d A-B$.
Then apply Proposition 4.1 of \cite{bbf}, where it was shown
that if $\Delta(E)\lhd A-B$ for some set
$E$ of positive upper Banach density then $A-B$ is piecewise Bohr.
\end{proof}

\medskip
As a final remark, we point out that the nonstandard methods used in this
paper for sets of integers, also work in more abstract
settings. In consequence, many of the results presented here
can be extended to the general framework of \emph{amenable groups}
(see \cite{dl}).

\bigskip
\textbf{Acknowledgments.}
I thank Vitaly Bergelson for his MiniDocCourse
held in Prague and Borov\'a Lada (Czech Republic) in January 2007:
those lectures were mostly responsible for my interest in
combinatorics of numbers; a grateful thank is due to the organizer
Jaroslav N\u{e}set\u{r}il, who generously supported my participation.
I also thank Mathias Beiglb\"ock for useful discussions
about the original proof of Jin's theorem, and for pointing out
the problem of finding a bound to the number
of shifts $(A-B)+t$ which are needed to cover a thick set.

\end{document}